\documentclass[12pt]{amsart}

\usepackage{amsthm,amssymb,amsmath,a4wide,graphicx,tikz}
\usepackage{subfigure}
\usepackage{verbatim}

\newtheorem{prop}{Proposition}[section]
\newtheorem{conjecture}{Conjecture}[section]
\newtheorem{theorem}{Theorem}[section]
\newtheorem{lemma}{Lemma}[section]

\newtheorem{remark}{Remark}[section]
\theoremstyle{definition}

\begin{document}
\title{The Random Continued Fraction Transformation}
\author{Charlene Kalle}
\address{Charlene Kalle: Mathematical Institute, University of Leiden, PO Box 9512, 2300 RA Leiden, The Netherlands} 
\email{kallecccj@math.leidenuniv.nl}
\thanks{Acknowledgements: The first author was supported by the NWO Veni-grant 639.031.140. The second author was supported by EPSRC grant EP/K029061/1.
}

\author{Tom Kempton}
\address{Tom Kempton: School of Mathematics and
Statistics
North Haugh
St Andrews
KY16 9SS
United Kingdom}
\email{tmwk@st-andrews.ac.uk}

\author{Evgeny Verbitskiy}
\address{Evgeny Verbitskiy: Mathematical Institute, University of Leiden, PO Box 9512, 2300 RA Leiden, The Netherlands \newline\indent \textup{and} \newline\indent Johann Bernoulli Institute for Mathematics and Computer Science, University of Groningen, PO Box 407, 9700 AK, Groningen, The Netherlands} 
\email{e.a.verbitskiy@rug.nl}


\begin{abstract}
We introduce a random dynamical system related to continued fraction expansions. It uses random combination of the Gauss map and the R\'enyi (or backwards) continued fraction map. We explore the continued fraction expansions that this system produces as well as the dynamical properties of the system.
\end{abstract}

\maketitle

\section{Introduction}
In 1913 (\cite{Per13}) Perron described an algorithm producing finite and infinite continued fraction expansions of real numbers of the form
\[ x =  d_0 + \cfrac{\epsilon_0}{d_1 + \cfrac{\epsilon_1}{d_2 + \ddots + \cfrac{\epsilon_{n-1}}{d_n+\ddots}}}, \]
where $d_0 \in \mathbb Z$ and for each $n \ge 1$, $\epsilon_{n-1}\in\{-1,1\}$, $d_n\in\mathbb N$ and $d_n + \epsilon_n \ge 1$. Moreover, in case the continued fraction is infinite, the algorithm guarantees that $d_n + \epsilon_{n+1} \ge 1$ infinitely often. Perron called these expansions semi-regular continued fractions. 

Within this framework one can see a number of more familiar systems of continued fractions, each of which can be studied by a corresponding dynamical system. Regular continued fractions, which correspond to letting each $\epsilon_n=1$, are generated by the Gauss map $Tx = \frac1{x}$ (mod 1). Backwards continued fractions, which were introduced by R\'enyi in \cite{Ren57}, correspond to $\epsilon_n=-1$. Odd and even continued fractions (\cite{HK02}) and $\alpha$- continued fractions (\cite{Nak81}) can also be seen within this framework. In this article we define a random dynamical system which allows one to generate all semi-regular continued fraction expansions of the form (\ref{q:cfeqn}) for any given $x$  and to study their dynamical and ergodic properties. We do not require the condition that $d_n + \epsilon_{n+1} \ge 1$ infinitely often, which makes our set-up slightly more general.


Over the last decade there has been a great deal of work on random dynamical systems. A random system is given by a finite family of maps defined on the same state space and a probabilistic regime for choosing one of these maps at each time step. The study of conditions that guarantee the existence of an invariant measure for such systems was initiated by Morita (\cite{Mor85}) and Pelikan (\cite{Pel84}). They consider the case where each map in the random system is piecewise smooth with respect to some finite partition and expanding on average. In \cite{GB03} and \cite{BG05} these results are extended to when the probabilistic regime is position dependent. These results were further generalised by Inoue (\cite{Ino12}) to more general underlying partitions, including ones with countably many elements. In \cite{ANV14} limit theorems are studied for random systems consisting of countably many maps. Their examples include families of maps that are piecewise smooth with respect to some finite partition and are expanding on average. Rousseau and Todd studied hitting time statistics for random maps \cite{RT}.

Random dynamical systems have also been used in relation to representations of numbers, see for example \cite{Mor}. In particular, the introduction of the random $\beta$-transformation has opened up new approaches to studying the dynamical and ergodic properties of $\beta$-expansions and Bernoulli convolutions, see \cite{DdV05,DK13,Kem13}. A key part of this approach was to study the invariant measures and ergodic properties of the random $\beta$-transformation, as was done in \cite{DdV05,DdV07,Kem14}. 

In this article we first introduce the random continued fraction map $K$. We show that it generates convergent continued fraction expansions for all points in its domain and we explore some of the properties of these expansions. The map $K$, and a related map $R$ which is easier to analyse, present several challenges which are interesting from a purely dynamical point of view. In particular, $R$ has countably many discontinuities and one of the interval maps defining $R$ has an indifferent fixed point. To overcome these difficultieswe build on the work of Inoue \cite{Ino12}, who studied transfer operators for countably branched skew product systems that are expanding on average. We can then use the results from \cite{ANV14} to obtain limit theorems. Other specific examples of random intermittent systems have been recently considered in \cite{BBD14}. Here Bahsoun, Bose and Duan studied limit theorems and mixing rates for random combinations of maps that are variations of the Manneville-Pomeau map.


The paper is outlined as follows. In Section 2 we first show that every expansion of the form (\ref{q:cfeqn}) is generated by $K$, and study further dynamical properties of random continued fractions. In Section 3 we prove that there exists an absolutely continuous invariant measure for $R$. The density of this measure is of bounded variation. A key part of our approach here is to study the transfer operator associated with $R$. We show that $R$ satisfies conditions of Inoue \cite{Ino12} which gives that the transfer operator is quasi-compact. In Section 4 we use this to show that the invariant measure is fully supported. We can then employ results from \cite{ANV14} to obtain that $R$ is mixing and that the Central Limit Theorem and Large Deviation Principle hold. Numerical evidence seems to suggest that the density in fact is quite smooth. In the last section we discuss this in more detail and also mention some open questions and future directions.


\section{The random map}

\subsection{Definition of the map}
It is clear that $x\in\mathbb R$ has an expansion of the form
\begin{equation}\label{q:cfeqn}
x =  \cfrac{\epsilon_0}{d_1 + \cfrac{\epsilon_1}{d_2 + \ddots + \cfrac{\epsilon_{n-1}}{d_n+\ddots}}},
\end{equation}
where $\epsilon_{n-1} \in\{-1,1\}$, $d_n\in\mathbb N$ and $d_n + \epsilon_n \ge 1$ for all $n \in \mathbb N$ if and only if $x \in [-1,1]\backslash \{0\}$. If $|x| >1$, we can subtract a suitable integer $d_0$ from $x$ and use the representation from (\ref{q:cfeqn}) for $x-d_0$ to obtain a continued fraction expansion of $x$.

To find the right definition of the random dynamical system $K$, we first ask which $x\in [-1,1]$ have expansions that begin with a given choice of $\epsilon_0$ and $d_1$. Let $\epsilon_0$ and $d_1>1$ be given. Then $x$ can be written in the form (\ref{q:cfeqn}) if and only if $x=\frac{\epsilon_0}{d_1+y}$ for some $y\in[-1,1]\setminus \{0\}$. This can be satisfied if and only if $\epsilon_0= sgn(x)$ and
\[
|x|=\epsilon_0 x\in\left(\dfrac{1}{d_1+1}, \dfrac{1}{d_1-1}\right).
\]
This typically gives two choices of the digit $d_1$. We see that
\[
y=\dfrac{\epsilon_0}{x}-d_1=\frac{1}{|x|}-d_1.
\]
Similarly, if $d_1=1$ then we require $\epsilon_1=1$ and have $x=\frac{\epsilon_0}{d_1+y}$ for some $y\in(0,1]$. This can be satisfied for $|x|\in(\frac{1}{2},1]$ and we have
\[
y=\dfrac{1}{|x|}-1.
\]
As is standard with dynamical constructions of expansions of real numbers, the possible values of $\epsilon_1, d_2$ are equal to the values of $\epsilon_0, d_1$ associated with $y=\frac{1}{|x|}-d_1$. Thus we can generate all expansions of the form (\ref{q:cfeqn}) using the following transformation.

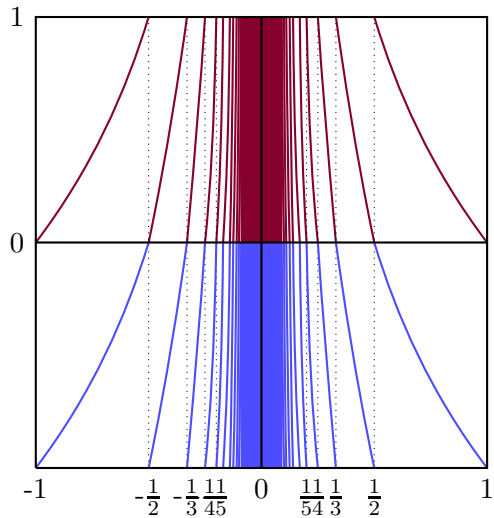
\begin{figure}[h]
\begin{center}
\begin{tikzpicture}[scale=3]
\filldraw[fill=purple!70!black, draw=purple!70!black] (0,0) rectangle (.09,1);
\filldraw[fill=purple!70!black, draw=purple!70!black] (0,0) rectangle (-.09,1);
\filldraw[fill=blue!70, draw=blue!70] (0,-1) rectangle (.09,0);
\filldraw[fill=blue!70, draw=blue!70] (0,-1) rectangle (-.09,0);
\draw[thick, purple!70!black] (1,0) .. controls (.75,.33) and (.6,.67) .. (.5,1);
\draw[thick, purple!70!black] (.5,0) .. controls (.42,.38) and (.36,.78) .. (.33,1);
\draw[thick, purple!70!black] (.33,0) .. controls (.29,.45) and (.27,.7) .. (.25,1);
\draw[thick, purple!70!black] (.25,0) .. controls (.225,.25) and (.21,.5) .. (.2,1);
\draw[thick, purple!70!black] (.2,0) .. controls (.185,.25) and (.175,.5) .. (.17,1);
\draw[thick, purple!70!black] (.17,0) .. controls (.155,.25) and (.145,.5) .. (.14,1);
\draw[thick, purple!70!black] (.14,0) .. controls (.133,.25) and (.13,.5) .. (.125,1);
\draw[thick, purple!70!black] (.125,0) .. controls (.118,.25) and (.115,.5) .. (.11,1);
\draw[thick, purple!70!black] (.11,0) .. controls (.105,.25) and (.102,.5) .. (.1,1);
\draw[thick, purple!70!black] (.1,0) .. controls (.095,.25) and (.092,.5) .. (.09,1);
\draw[thick, purple!70!black] (-1,0) .. controls (-.75,.33) and (-.6,.67) .. (-.5,1);
\draw[thick, purple!70!black] (-.5,0) .. controls (-.42,.38) and (-.36,.78) .. (-.33,1);
\draw[thick, purple!70!black] (-.33,0) .. controls (-.29,.45) and (-.27,.7) .. (-.25,1);
\draw[thick, purple!70!black] (-.25,0) .. controls (-.225,.25) and (-.21,.5) .. (-.2,1);
\draw[thick, purple!70!black] (-.2,0) .. controls (-.185,.25) and (-.175,.5) .. (-.17,1);
\draw[thick, purple!70!black] (-.17,0) .. controls (-.155,.25) and (-.145,.5) .. (-.14,1);
\draw[thick, purple!70!black] (-.14,0) .. controls (-.133,.25) and (-.13,.5) .. (-.125,1);
\draw[thick, purple!70!black] (-.125,0) .. controls (-.118,.25) and (-.115,.5) .. (-.11,1);
\draw[thick, purple!70!black] (-.11,0) .. controls (-.105,.25) and (-.102,.5) .. (-.1,1);
\draw[thick, purple!70!black] (-.1,0) .. controls (-.095,.25) and (-.092,.5) .. (-.09,1);
\draw[dotted](.2,-1)--(.2,1)(.25,-1)--(.25,1)(.33,-1)--(.33,1)(.5,-1)--(.5,1)(-.2,-1)--(-.2,1)(-.25,-1)--(-.25,1)(-.33,-1)--(-.33,1)(-.5,-1)--(-.5,1);

\draw[thick, blue!70] (1,-1) .. controls (.75,-.67) and (.6,-.33) .. (.5,0);
\draw[thick, blue!70] (.5,-1) .. controls (.42,-.62) and (.36,-.22) .. (.33,0);
\draw[thick, blue!70] (.33,-1) .. controls (.29,-.55) and (.27,-.3) .. (.25,0);
\draw[thick, blue!70] (.25,-1) .. controls (.225,-.75) and (.21,-.5) .. (.2,0);
\draw[thick, blue!70] (.2,-1) .. controls (.185,-.75) and (.175,-.5) .. (.17,0);
\draw[thick, blue!70] (.17,-1) .. controls (.155,-.75) and (.145,-.5) .. (.14,0);
\draw[thick, blue!70] (.14,-1) .. controls (.133,-.75) and (.13,-.5) .. (.125,0);
\draw[thick, blue!70] (.125,-1) .. controls (.118,-.75) and (.115,-.5) .. (.11,0);
\draw[thick, blue!70] (.11,-1) .. controls (.105,-.75) and (.102,-.5) .. (.1,0);
\draw[thick, blue!70] (.1,-1) .. controls (.095,-.75) and (.092,-.5) .. (.09,0);
\draw[thick, blue!70] (-1,-1) .. controls (-.75,-.67) and (-.6,-.33) .. (-.5,0);
\draw[thick, blue!70] (-.5,-1) .. controls (-.42,-.62) and (-.36,-.22) .. (-.33,0);
\draw[thick, blue!70] (-.33,-1) .. controls (-.29,-.55) and (-.27,-.3) .. (-.25,0);
\draw[thick, blue!70] (-.25,-1) .. controls (-.225,-.75) and (-.21,-.5) .. (-.2,0);
\draw[thick, blue!70] (-.2,-1) .. controls (-.185,-.75) and (-.175,-.5) .. (-.17,0);
\draw[thick, blue!70] (-.17,-1) .. controls (-.155,-.75) and (-.145,-.5) .. (-.14,0);
\draw[thick, blue!70] (-.14,-1) .. controls (-.133,-.75) and (-.13,-.5) .. (-.125,0);
\draw[thick, blue!70] (-.125,-1) .. controls (-.118,-.75) and (-.115,-.5) .. (-.11,0);
\draw[thick, blue!70] (-.11,-1) .. controls (-.105,-.75) and (-.102,-.5) .. (-.1,0);
\draw[thick, blue!70] (-.1,-1) .. controls (-.095,-.75) and (-.092,-.5) .. (-.09,0);

\draw[thick](-1,-1)node[below]{\small -1}--(-.5,-1)node[below]{\small -$\frac12$}--(-.33,-1)node[below]{\small -$\frac13$}--(-.25,-1)node[below]{\small -$\frac14$}--(-.2,-1)node[below]{\small -$\frac15$}--(0,-1)node[below]{\small 0}--(.2,-1)node[below]{\small $\frac15$}--(.25,-1)node[below]{\small $\frac14$}--(.33,-1)node[below]{\small $\frac13$}--(.5,-1)node[below]{\small $\frac12$}--(1,-1)node[below]{\small 1}--(1,1)--(-1,1)node[left]{\small 1}--(-1,0)node[left]{\small 0}--(-1,-1)(-1,0)--(1,0)(0,-1)--(0,1);
\end{tikzpicture}
\caption{The map $T_0$ in red and $T_1$ in blue.}
\label{f:complete}
\end{center}
\end{figure}

Let $\Omega = \{ (\omega_k)_{k \ge 1} \, : \, \omega_k \in \{0,1\} \} = \{ 0,1 \}^{\mathbb N}$ and let $\sigma:\Omega\to\Omega$ be the left shift. We define the random continued fraction map $K: \Omega \times [-1,1] \to \Omega \times [-1,1]$ by setting $K(\omega, 0)=(\sigma(\omega),0)$ and for $|x| \in \big( \frac1{k+1}, \frac1{k} \big]$,
\[ K(\omega, x) = \Big(\sigma(\omega), \Big| \frac1x \Big| - (k + \omega_1) \Big).\]

One can think of the map $K$ as follows. Let maps $T_0, T_1 : [0,1] \to [0,1]$ be the Gauss and R\'enyi maps respectively, given by
\begin{equation}\label{q:gaussrenyi}
T_0 x = \left\{
\begin{array}{ll}
0, & \text{if } x =0,\\
\\
\displaystyle \frac1x \,\text{ (mod 1)}, & \text{otherwise},
\end{array}\right.
\quad \text{and} \quad T_1 x = \left\{
\begin{array}{ll}
0, & \text{if } x=1,\\
\\
\displaystyle \frac1{1-x} \, \text{ (mod 1)}, & \text{otherwise}.
\end{array}\right.
\end{equation}
Let $\pi: \Omega \times [-1,1]\to [-1,1]$ be given by $\pi(\omega,x)=x$. Then
\[ \pi \big( K(\omega,x) \big) = \left\{
\begin{array}{ll}
T_0 x - \omega_1, & \text{if } x > 0,\\
\\
T_1 (x+1) - \omega_1, & \text{if } x < 0.\\
\end{array}\right.\]
Set $\omega_0 = 0$ if $x \ge 0$ and 1 otherwise. This gives
\begin{equation}\label{q:pi-K}
\pi \big( K(\omega,x) \big) =  T_{\omega_0} (x + \omega_0) - \omega_1.
\end{equation}
Note that $\pi\big(K(\omega,x)\big) \in [0,1]$ if $\omega_1=0$ and $\pi\big(K(\omega,x)\big) \in [-1,0]$ if $\omega_1=1$. By iterating we see that for each $n \ge 1$ and all $(\omega,x) \in \Omega \times [-1,1]$, such that $\pi \big( K^m (\omega,x) \big) \neq 0$ for $0 \le m <n$, we have
\[ \pi \big( K^n (\omega,x) \big) = ( T_{\omega_{n-1}} \circ \cdots \circ T_{\omega_0} )(x + \omega_0) - \omega_n.\]

\subsection{Random continued fraction expansions}

For $n=1$, set
\[ d_1=d_1(\omega, x) = \left\{
\begin{array}{ll}
k + \omega_1,& \text{if } |x| \in \Big( \frac{1}{k+1} , \frac1k \Big],\\
\\
\infty, & \text{if } x=0,
\end{array} \right.
\]
and for $n \ge 2$, define $d_n(\omega, x) = d_1\big(K^{n-1}(\omega, x)\big)$. We have
\begin{equation}\label{q:K}
\pi \big(K(\omega, x) \big) = (-1)^{\omega_0} \frac1x - d_1.
\end{equation}
and for each $n \ge 1$ such that $\pi(K^m(\omega, x)) \neq 0$ for all $0 \le m \le n$,
\[ x = \frac{(-1)^{\omega_0}}{d_1+\pi(K(\omega, x))} = \cfrac{(-1)^{\omega_0}}{d_1+\cfrac{(-1)^{\omega_1}}{d_2+\pi(K(\omega, x))}} = \cdots =  \cfrac{(-1)^{\omega_0}}{d_1 + \cfrac{(-1)^{\omega_1}}{d_2 + \ddots + \cfrac{(-1)^{\omega_{n-1}}}{d_n+\pi(K^n(\omega,x))}}}.\]

The digit sequence $\big( d_n (\omega,x) \big)_{n \ge 1}$ represents the continued fraction representation of the pair $(\omega,x)$ as given by $K$. If there is a smallest integer $n$ such that $d_n(\omega,x)=\infty$, then $(\omega,x)$ has {\bf finite random continued fraction expansion}
\[x =  \cfrac{(-1)^{\omega_0}}{d_1 + \cfrac{(-1)^{\omega_1}}{d_2 + \ddots + \cfrac{(-1)^{\omega_{n-2}}}{d_{n-1}}}}.\]
Now suppose that $d_n(\omega,x)$ is finite for all $n \ge 1$. We want to show that $(\omega,x)$ has {\bf infinite random continued fraction expansion}
\[ x =  \cfrac{(-1)^{\omega_0}}{d_1 + \cfrac{(-1)^{\omega_1}}{d_2 + \ddots + \cfrac{(-1)^{\omega_{n-1}}}{d_n+\ddots}}}.\]
For each $n \ge 1$, let $\frac{p_n}{q_n}$ denote the convergents of the random continued fraction of $(\omega,x)$, i.e., write
\[ \frac{p_n}{q_n} = \cfrac{(-1)^{\omega_0}}{d_1 + \cfrac{(-1)^{\omega_1}}{d_2 + \ddots + \cfrac{(-1)^{\omega_{n-1}}}{d_n}}}.\]
\subsection{Convergence of Partial Sums}
To show convergence, we follow the approach for regular continued fractions with the important difference that in our case the numbers $q_n$ do not necessarily increase. As for regular continued fractions (see for example \cite{DK02}), one can obtain the following relations:
\begin{align*}
p_{-1} =1,\quad p_0=0, & \hspace{.5cm} p_n=d_n p_{n-1}+(-1)^{\omega_{n-1}} p_{n-2},\\
q_{-1}=0, \quad q_0=1, & \hspace{.5cm} q_n = d_n q_{n-1} + (-1)^{\omega_{n-1}} q_{n-2}.
\end{align*}
Using these recurrences, induction easily gives that
\begin{equation}\label{q:xfrac}
x = \frac{p_n+p_{n-1}\pi \big( K^n(\omega,x) \big)}{q_n+q_{n-1}\pi \big( K^n(\omega,x) \big)}
\end{equation}
and that
\begin{equation}\label{q:pnqnminus}
p_{n-1}q_n - p_nq_{n-1} = (-1)^n(-1)^{\omega_0 + \cdots + \omega_{n-1}}.
\end{equation}
The next lemmas are needed to show that although the sequence $\{q_n\}_{n \ge 1}$ is not necessarily increasing, we still have $\lim_{n \to \infty} \frac1{q_n}=0$.

\begin{lemma}\label{l:qnsmall}
For each $n \ge 2$, $q_n >0$ and if $q_n \le q_{n-1}$, then $\omega_{n-1}=1$ and $d_n=1$, so $\omega_n=0$.
\end{lemma}

\begin{proof}
We prove this by induction. For $n=1$ we have $q_1 = d_1 \ge 1 = q_0$. For $n=2$, we have $q_2 = d_2d_1 + (-1)^{\omega_1}$ and hence
\[ q_2 \le q_1 \quad \Leftrightarrow \quad d_2 \le 1-\frac{(-1)^{\omega_1}}{d_1}.\]
This can happen only if $\omega_1=1$, thus $d_1>1$, and $d_2=1$. This gives $q_2 = d_1-1>0$ and the lemma. Now suppose the statements hold for all $k < n$, i.e., $q_k >0$  and if $q_k \le q_{k-1}$, then $\omega_{k-1}=1$ and $d_k=1$, so $\omega_k=0$. If $q_n > q_{n-1}$, then automatically $q_n >0$. So, suppose $q_n \le q_{n-1}$. We have
\[ q_n = d_n q_{n-1} + (-1)^{\omega_{n-1}}q_{n-2} \le q_{n-1} \quad \Leftrightarrow \quad 1 \le d_n \le 1-(-1)^{\omega_{n-1}}\frac{q_{n-2}}{q_{n-1}}.\]
Since $q_{n-2}, q_{n-1}>0$, $q_n \le q_{n-1}$ implies that $\omega_{n-1}=1$ and hence $d_{n-1}>1$. The induction hypothesis then gives that $\frac{q_{n-2}}{q_{n-1}} < 1$ and hence $1 \le d_n < 2$. This gives the lemma.
\end{proof}

\begin{lemma}\label{l:bigger}
If $q_n \le q_{n-1}$, then $q_{n-2}< q_{n-1}<q_{n+1}$. This implies that $q_n \neq q_{n-1}$.
\end{lemma}

\begin{proof}
By the previous lemma we have $\omega_{n-1}=1$, so $q_{n-1}>q_{n-2}$. We also have $d_n =1$ and thus $\omega_n=0$. Since $q_n >0$, this implies that $q_{n+1} = d_{n+1} q_n + q_{n-1} > q_{n-1}$. For the second part, note that if $q_n = q_{n-1}$, then $d_n = 1-(-1)^{\omega_{n-1}} \frac{q_{n-2}}{q_{n-1}}$ and by the previous lemma $q_{n-1}>q_{n-2}$. Since $d_n$ is an integer, this is not possible.
\end{proof}

Before we can prove that the process converges, we need a lower bound on the $q_n$'s in case $q_n < q_{n-1}$. This is done in the next lemma.
\begin{lemma}\label{l:estimates}
Suppose $q_n < q_{n-1}$, so $d_{n-1}>1$.
\begin{itemize}
\item[(i)] If $d_{n-1} > 2$, then $q_n > q_{n-2}$.
\item[(ii)] If $d_{n-1}=2$ and $(\omega_k,d_k) = (1,2)$ for all $1 \le k \le n-1$, then $q_n=1$ and $q_{n-1}=n$.
\item[(iii)] Suppose $d_{n-1}=2$ and there is a $1\le k < n-1$ such that $( \omega_k ,d_k) \neq (1,2)$. Let $k$ be the largest such index. Then $q_n > q_{k-1}$.
\end{itemize}
\end{lemma}

\begin{proof}
Recall that $q_n < q_{n-1}$ implies that $d_n=1$, $\omega_n=0$ and $\omega_{n-1}=1$, so $d_{n-1}>1$. Hence, $q_n = q_{n-1}-q_{n-2}$.\\
(i) If $d_{n-1} >2$, then 
\begin{eqnarray*}
q_n &=& d_{n-1} q_{n-2} + (-1)^{\omega_{n-2}}q_{n-3} -q_{n-2} = (d_{n-1}-1)q_{n-2} + (-1)^{\omega_{n-2}}q_{n-3}\\
& \ge & 2 q_{n-2} + (-1)^{\omega_{n-2}}q_{n-3}.
\end{eqnarray*}
If $q_{n-2} > q_{n-3}$, then $q_n \ge 2 q_{n-2} - q_{n-3} > q_{n-2}$. If $q_{n-2} < q_{n-3}$, then Lemma~\ref{l:qnsmall} gives that $\omega_{n-2}=0$ and hence $q_n \ge 2 q_{n-2} + q_{n-3} > q_{n-3} > q_{n-2}$.\\
For both (ii) and (iii), note that if $d_k=2$ and $\omega_k=1$ for some $1\le k \le n-1$, then $q_k = 2q_{k-1} -q_{k-2}$, so
\begin{equation}\label{q:qk-qk-1}
q_k - q_{k-1}=q_{k-1}-q_{k-2}.
\end{equation}
(ii) From (\ref{q:qk-qk-1}) it follows that
\[ q_n =q_{n-1}-q_{n-2} = q_2 -q_1 = 2\cdot 2 +(-1)^1 -2 =1.\]
Moreover, for each $1 \le k \le n-1$ it holds that
\[ q_k = 2q_{k-1}-q_{k-2} = q_{k-1} + (q_{k-1}-q_{k-2})=q_{k-1}+1.\]
Hence, $q_{n-1} = n-2+q_1=n$.\\
(iii) Let $k$ be as given in the lemma, so $(\omega_k,d_k) \neq (1,2)$ and $(\omega_j, d_j)=(1,2)$ for $k+1 \le j \le n-1$. Then,
\[ q_n = q_{n-1}-q_{n-2} = q_{k+1}-q_k = d_{k+1}q_k + (-1)^{\omega_k}q_{k-1}-q_k = q_k + (-1)^{\omega_k}q_{k-1}.\]
If $\omega_k =0$, then $q_n = q_k + q_{k-1}> q_{k-1}$. If $\omega_k=1$, then $d_k \ge 3$ and $q_k > q_{k-1}$ by Lemma~\ref{l:qnsmall}. This gives
\[ q_n =q_k - q_{k-1} = (d_k -1)q_{k-1} +(-1)^{\omega_{k-1}}q_{k-2} \ge 2q_{k-1} + (-1)^{\omega_{k-1}}q_{k-2}.\]
As in the proof of part (i) we now get that if $q_{k-1} > q_{k-2}$, then $q_n > q_{k-1}$ and if $q_{k-1} < q_{k-2}$, then $\omega_{k-1}=0$ and we also get $q_n > q_{k-1}$.
\end{proof}

\begin{prop}
Let $x \in [-1,1]\backslash \mathbb Q$. For each $\omega$, the digits $d_n(\omega, x)$ give a continued fraction expansion of $x$.
\end{prop}

\begin{proof}
For all $n \ge 1$ we have using (\ref{q:xfrac}) and (\ref{q:pnqnminus})
\begin{eqnarray*}
\big| x - \frac{p_n}{q_n} \big| &=& \Big| \frac{q_n \big( p_n + p_{n-1}\pi\big(K^n(\omega,x) \big) \big) - p_n \big( q_n + q_{n-1} \pi\big(K^n(\omega,x) \big) \big)}{ q_n \big(q_n + q_{n-1} \pi\big(K^n(\omega,x) \big) \big)} \Big|\\
&=& \Big| \frac{\pi\big(K^n(\omega,x) \big) (q_np_{n-1}-p_nq_{n-1})}{q_n \big(q_n + q_{n-1} \pi\big(K^n(\omega,x) \big) \big)} \Big|\\
& \le & \frac{1}{|q_n \big(q_n + q_{n-1} \pi\big(K^n(\omega,x) \big) \big)|} \le \frac1{q_n|q_n-q_{n-1}|} \le \frac1{q_n}.
\end{eqnarray*}
We now show that $\lim_{n \to \infty} \frac1{q_n}=0$. Let $\varepsilon >0$. By Lemma~\ref{l:bigger} there exists a subsequence $(q_{n_k})_{k \ge 0}$ such that $\lim_{k \to \infty} \frac1{q_{n_k}}=0$ and hence there exists an $N_1$ such that $\frac1{q_{N_1}} < \varepsilon$. If there is no $n > N_1$, such that $q_n < q_{n-1}$, then we are done. If there is, let $k$ be the smallest such index. If $( \omega_{N_1+1}, d_{N_1+1}) \neq (1,2)$, by Lemma~\ref{l:estimates} $q_k \ge q_{N_1}$ and then same holds for all other $n > N_1$. If $(\omega_{N_1+1}, d_{N_1+1}) = (1,2)$, then set $N=k+1$. We then have $\frac1{q_N} < \frac1{q_{k-1}} < \frac1{q_{N_1}} < \varepsilon$. Moreover, for all $n >N $ we have $q_n > q_{k-1}$, since $(\omega_k,d_k)=(0,1)$. This shows that the limit exists and is equal to 0. Hence we get
\[ x =  \cfrac{(-1)^{\omega_0}}{d_1 + \cfrac{(-1)^{\omega_1}}{d_2 + \cfrac{(-1)^{\omega_2}}{d_3+\ddots}}}\]
for each $x \in [-1,1]\backslash \mathbb Q$.
\end{proof}

\section{Invariant densities}
Recall that $T_0, T_1:[0,1] \to [0,1]$ are the Gauss and R\'enyi map respectively. To study the dynamical properties of random continued fractions we replace the map $K$ by the transformation $R: \Omega \times [0,1]\to \Omega \times [0,1]$ given by
\[ R(\omega,x) = (\sigma \omega, T_{\omega_1} x).\]
The ergodic properties of $R$ are easier to study and (\ref{q:pi-K}) gives the following simple relation between $K$ and $R$. Recall that $\omega_0=0$ if $x \ge 0$ and 1 otherwise. Then,
\[ \pi \big( R^{n-1} (\omega, T_0 x)\big) = \pi \big( K^n (\omega,x)\big) + \omega_n,\]
and hence,
\[
\Big| \pi \big( K^n (\omega,x)\big) \Big| = \omega_n + (-1)^{\omega_n} \pi \big( R^{n-1} (\omega, T_0 x)\big).
\]
Using this we can recover the digit sequence $\big( d_n(\omega,x)\big)_{n \ge 1}$ generated by $K$ as follows. For $(\omega,x) \in \Omega \times [0,1]$ define
\begin{equation}\label{q:bi}
b(\omega,x) = \left\{
\begin{array}{ll}
 k+ \omega_2, & \text{if } \omega_1 + (-1)^{\omega_1} x \in \big( \frac1{k+1}, \frac1k \big],\\
 \infty, & \text{if } \omega_1 + (-1)^{\omega_1} x =0.
\end{array}\right.
\end{equation}
Write $0\omega$ for the sequence $\omega' \in \Omega$ satisfying $\omega'_1=0$ and $\omega'_{n+1} = \omega_n$ for all $n \ge 1$. Then for $n \ge 1$,
\begin{equation}\label{q:dandb}
d_n(\omega, x) = b \big( R^{n-1} (0\omega, x)\big).
\end{equation}

\begin{remark}{\rm

The Gauss map $T_0$ is intimately related to the well known Farey map, which is defined by
\[ F x = \left\{
\begin{array}{ll}
\displaystyle \frac{x}{1-x}, & \text{if } 0 \le x \le \frac12,\\
\\
\displaystyle \frac{1-x}{x}, & \text{if } \frac12 \le x \le 1.
\end{array}\right.\]
The Gauss map can be obtained from $F$ by inducing on the first passage time to the interval $\big[ \frac12, 1 \big]$. For $x \in [0,1]$ and $n \ge 1$, set $\omega_n=1$ if $\pi \big( R^{n-1} (\omega,x) \big) < \frac12$ and $\omega_n=0$ if $\pi \big( R^{n-1} (\omega,x) \big) > \frac12$. If $\pi \big( R^{n-1} (\omega,x) \big) = \frac12$, you can choose. Then $\pi \big( R^n (\omega,x) \big) = F^nx$ for each $n \ge 1$.
}\end{remark}

\subsection{Existence}
Let $\lambda$ denote the Lebesgue measure on $[0,1]$. We are going to show that $R$ has invariant measures of type $m_p \otimes \mu_p$, where $m_p$ is the $(p,1-p)$-Bernoulli measure on $\Omega$ and $\mu_p$ is a probability measure on $[0,1]$ absolutely continuous with respect to $\lambda$. For this we use \cite{Ino12}. Recall that a function $g:[0,1]\to[0,1]$ is said to have bounded variation if
\[
\bigvee_{[0,1]}g:=\sup_{0=x_0<x_1<x_2<\cdots<x_n=1} \sum_{i=1}^n |g(x_i)-g(x_{i-1})|<\infty.
\]
The following is a very simplified, discrete version of the main theorem of \cite{Ino12}:
\begin{theorem}[Inoue, \cite{Ino12}]\label{Inoue}
Given two non-singular maps $T_0, T_1:[0,1]\to[0,1]$, let $R:\{0,1\}^{\mathbb N}\times[0,1]\to\{0,1\}^{\mathbb N}\times[0,1]$ be given by
\[
R(\omega,x)=(\sigma(\omega), T_{\omega_1}(x)).
\]
Let $ p \in [0,1]$ and set $p_0=p$ and $p_1 = 1-p$. For $ i \in \{0,1\}$ let $\{I_{i,k}\}$ be a countable partition of $[0,1]$ into intervals and use $\text{int}(I_{i,k})$ to denote the interior of these intervals. Let $g(i,x)$ be functions satisfying
\begin{equation}\label{q:g}
g(i,x)= \frac{p_i}{|T_i'(x)|}
\end{equation}
on $\bigcup_k \text{int}(I_{i,k}) $. Assume that the following conditions are satisfied:
\begin{enumerate}
 \item[(I1)] The restrictions of $T_i$ to each interval $\text{int}(I_{i,k})$ are $C^1$ and monotone.
 \item[(I2)] The weighted average expansion of $T_i$ is uniformly positive for all $x$, i.e.,
\[ \sup_{x\in[0,1]} \big( g(0,x) + g(1,x) \big) <1.\]
\item[(I3)] For each $i\in\{0,1\}$ the functions $g(i,x):[0,1]\to\mathbb R$ are of bounded variation.
\end{enumerate}
Then there exists a probability measure $\mu_p$ on $[0,1]$ absolutely continuous with respect to the Lebesgue measure $ \lambda$ with density function $h_p$ that is of bounded variation. Moreover, $\mu_p$ has the property that
\[ \mu_p (A) = p \mu_p (T_0^{-1}A) + (1-p)\mu_p (T_1^{-1}A) \]
for each Borel measurable set $A \subseteq [0,1]$.
\end{theorem}

We have not included condition $(A2)$ of Inoue since this automatically holds when one only has a finite number of maps $T_i$. We apply Theorem~\ref{Inoue} to our setting. 
\begin{prop}\label{p:conditions}
Suppose $p \in (0,1)$. Then the maps $T_0$ and $T_1$ satisfy the conditions (I1), (I2) and (I3) of Theorem \ref{Inoue}.
\end{prop}

\begin{proof}
For $k \ge 1$, let $I_{0,k} = \big( \frac1{k+1}, \frac1{k} \big]$ and $I_{1,k} = \big[ \frac{k-1}{k}, \frac{k}{k+1} \big)$. Consider the following partitions of $[0,1]$:
\begin{equation}\label{q:partition}
\mathcal I_0 = \{I_{0,k}\}_{k \ge 1} \quad \text{ and } \mathcal I_1 = \{I_{1,k}\}_{k \ge 1}.
\end{equation}
Then $T_0$ and $T_1$ are both $C^1$ and monotone on the interiors of the intervals of their respective partitions, so condition (I1) is satisfied. Note that the functions $g(i,\cdot)$ from (\ref{q:g}) become
\[g(0,x)= px^2 \quad \text{ and } \quad g(1,x) = (1-p)(1-x)^2.\]
almost everywhere. We see that
\[ \sup_{x \in [0,1]} \big( g(0,x) + g(1,x) \big) = \sup_{x \in [0,1]} \big( x^2 -2(1-p)x + (1-p) \big)= \max \{ 1-p, p \}. \]
So condition (I2) is satisfied for all $p \in (0,1)$. Since both $g(0, \cdot)$ and $g(1, \cdot)$ are monotone functions on the interval $[0,1]$, we have
\[ \bigvee_{[0,1]}g (0, \cdot) = p \quad \text{ and } \bigvee_{[0,1]}g (1, \cdot) = 1-p.\]
This gives (I3), and the proof
\end{proof}

The conclusions of Theorem~\ref{Inoue} yield the following result, the proof of which is standard, but included for the convenience of the reader.
\begin{theorem}\label{t:invm}
For any choice of parameter $0< p < 1$, there is an absolutely continuous probability measure $\mu_p \ll \lambda$ such that the product measure $m_p\otimes \mu_p$ is invariant for $R$. The probability density function $h_p$ of the measure $\mu_p$ is of bounded variation.
\end{theorem}

\begin{proof}
Theorem~\ref{Inoue} gives an absolutely continuous measure $\mu_p$ with the property that for each Borel set $A \subset [0,1]$, we have
\begin{equation}\label{q:mup}
\mu_p (A) = p \mu_p (T_0^{-1}A) + (1-p)\mu_p (T_1^{-1}A).
\end{equation}
Take a cylinder $[j_1 \cdots j_n] \in \{ 0,1\}^{\mathbb N}$ and an interval $(a,b) \subset [0,1]$. Then
\[ R^{-1} \big( [j_1 \cdots j_n] \times (a,b) \big)
= [0j_1 \cdots j_n] \times T_0^{-1}\big((a,b) \big) \cup [1j_1 \cdots j_n] \times T_1^{-1}\big((a,b) \big).\]
Hence,
\begin{align*}
(m_p \otimes \mu_p)\Big( & R^{-1} \big( [j_1 \cdots j_n] \times (a,b) \big) \Big)\\
&= p \, m_p ([j_1 \cdots j_n]) \mu_p \Big(T_0^{-1}\big( (a,b) \big) \Big) + (1-p) \, m_p ([j_1 \cdots j_n]) \mu_p \Big(T_1^{-1}\big( (a,b) \big) \Big)\\
&= m_p ([j_1 \cdots j_n]) \mu_p \big((a,b) \big) =(m_p \otimes \mu_p) \big( [j_1 \cdots j_n] \times (a,b) \big).
\end{align*}
This gives the first part of the result. The density $h_p$ is given by Theorem~\ref{Inoue}.
\end{proof}

If $p=1$, then $R$ reduces to the Gauss map and $h_1(x) = \frac1{\log 2}\frac1{x+1}$. If, on the other hand, we take $p=0$, then $R$ reduces to the map $T_1$, which has no absolutely continuous invariant probability measure, but does have an infinite and $\sigma$-finite absolutely continuous invariant measure with density $h_0 (x)=\frac1{x}$. This was proved by R\'enyi in \cite{Ren57}. The fact that $R$ is expanding on average causes that $h_p$ is a probability density for all $0<p<1$. It would be interesting to analyse the behaviour of $h_p$ as $p \to 0$.

\subsection{Properties of the invariant measure}
In this section we list some of the properties of the invariant measure $m_p \otimes \mu_p$. In \cite{Ino12} Inoue proved the existence of the density $h_p$ by analysing a random version of the Perron-Frobenius operator. In our case, the corresponding operator $\mathcal L_p: L^1(\lambda) \to L^1(\lambda)$ is given by
\begin{equation}\label{q:pf}
(\mathcal L_p f)(x) = \sum_{k \ge 1} \Big[ \frac{p}{(k+x)^2} f \Big( \frac1{k+x} \Big) + \frac{1-p}{(
k+x)^2} f \Big( 1-\frac1{k+x} \Big) \Big].
\end{equation} 
Theorem \ref{Inoue} is proved in \cite{Ino12} by showing that this operator has a fixed point in the space of functions of bounded variation, which is our function $h_p$. Recall that functions of bounded variation can be modified on a countable number of points to obtain a lower semi-continuous function. From now on we assume that $h_p$ is lower semi-continuous. Under the conditions of Theorem~\ref{Inoue} the operator $\mathcal L_p$ is quasi-compact and constrictive. In this section we use these results to derive some dynamical properties of $R$. We first prove that $R$ satisfies a strong Random Covering Property.

\begin{prop}\label{p:rcp}
Let $I \subseteq [0,1]$ be a non-trivial interval. Then for every $\omega \in \Omega$, there is an $n \ge 1$ such that
\[ \big( T_{\omega_n} \circ \cdots \circ T_{\omega_1}\big) I = [0,1).\]
\end{prop}

\begin{proof}
Recall the definition of the partitions $\mathcal I_0$ and $\mathcal I_1$ in the proof of Proposition~\ref{p:conditions}. If a non-trivial interval $J$ is contained in one of the intervals in $\mathcal I_i$, then $T_i J$ is again an interval and $\lambda (T_i J) > \lambda(J)$. Hence, there is an $m\ge 1$, such that $\big( T_{\omega_m} \circ \cdots \circ T_{\omega_1}\big) I$ contains an endpoint of one of the intervals in $\mathcal I_{\omega_{m+1}}$. This means that $\big(T_{\omega_{m+1}} \circ T_{\omega_m} \circ \cdots \circ T_{\omega_1}\big) I$ contains an interval of the form $[0,c)$ and an interval of the form $(1-c,1)$ for some $c>0$. Thus $\big(T_{\omega_{m+2}} \circ \cdots \circ T_{\omega_1}\big) I = [0,1)$.
\end{proof}

It now follows that the measure $\mu_p$ is in fact equivalent to $\lambda$. The proof below is essentially the one from \cite{ANV14}.
\begin{prop}
Let $h_p$ be the probability density function from Theorem~\ref{t:invm}. Then $h_p(x)>0$ for all $x \in [0,1)$.
\end{prop}

\begin{proof}
We know that $h_p$ is a function of bounded variation satisfying $h_p \ge 0$, $\int_{[0,1]} h_p d\mu_p =1$ and $\mathcal L_p h_p=h_p$. Therefore there exists a non-trivial interval $I \subseteq[0,1]$ and an $\alpha >0$, such that $h_p \ge \alpha1_{I}$. Fix a sequence $\bar \omega = (\bar \omega_1, \bar \omega_2, \ldots) \in \Omega$. Then for any $n \ge 1$,
\begin{eqnarray*}
h_p(x) &=& \mathcal L^n_p h_p(x) \, \ge \, \alpha \mathcal L^n_p 1_I (x)\\
& = & \sum_{(\omega_1, \ldots, \omega_n) \in \Omega^n} \sum_{y \in (T_{\omega_n}\circ \cdots \circ T_{\omega_1})^{-1}\{x\}} \frac{p_{\omega_1}\cdots p_{\omega_n}}{|(T_{\omega_n}\circ \cdots \circ T_{\omega_1})'y|} 1_I (y)\\
& \ge & \sum_{y \in (T_{\bar \omega_n} \circ \cdots \circ T_{\bar \omega_1})^{-1}\{x\}} \frac{p_{\bar \omega_1}\cdots p_{\bar \omega_n}}{|(T_{\bar \omega_n} \circ \cdots \circ T_{\bar \omega_1})'y|} 1_I(y).
\end{eqnarray*}
By Proposition~\ref{p:rcp} there is an $n \ge 1$ such that for all $x \in [0,1)$,
\[ (T_{\bar \omega_n} \circ \cdots \circ T_{\bar \omega_n})^{-1}\{x\} \cap I \neq \emptyset.\]
Hence, for all $0 < p < 1$ and all $x \in [0,1)$, $h_p(x) >0$.
\end{proof}

This proposition leads to the following result, the proof of which uses a standard technique that can be found for example in \cite{BG97}.
\begin{prop}\label{p:boundedaway}
The probability measure $h_p$ is bounded from above and away from 0.
\end{prop}

\begin{proof}
The fact that $h_p$ is bounded from above follows since $h_p$ is of bounded variation. In the previous proposition, we have established that $h_p(x)>0$ for all $x \in [0,1)$. Since $h_p$ is lower semi-continuous, it takes its minimum on [0,1]. Therefore, it is enough to show that $h_p(1)>0$. Let $\varepsilon >0$ be small and let $k \ge 1$. We consider part of the inverse image of the interval $(1-\varepsilon, 1)$ under $T_0$ (we could just as well use $T_1$). Note that
\[  \Big( \frac1{k+1}, \frac1{k+1-\varepsilon} \Big) \subseteq T_0^{-1} ( 1-\varepsilon, 1)\]
and that
\[ \lambda \Big(\Big( \frac1{k+1}, \frac1{k+1-\varepsilon} \Big) \Big)= \frac{\varepsilon}{(k+1)(k+1-\varepsilon)}.\]
Hence,
\[ k^2 \lambda \Big(\Big( \frac1{k+1}, \frac1{k+1-\varepsilon} \Big) \Big) < \lambda \big( (1-\varepsilon,1) \big) < (k+1)^2 \lambda \Big(\Big( \frac1{k+1}, \frac1{k+1-\varepsilon} \Big) \Big).\]
It then follows that
\begin{eqnarray*}
\lim_{x \uparrow 1} h_p(x) &=& \lim_{\varepsilon \to 0} \frac1{\lambda \big( (1-\varepsilon,1)\big)} \int_{1-\varepsilon}^1 h_p(x) dx \, = \, \lim_{\varepsilon \to 0} \frac{\mu_p\big((1-\varepsilon, 1)\big)}{\lambda \big( (1-\varepsilon,1)\big)}\\
& \ge & \lim_{\varepsilon \to 0} \frac{p \mu_p \big( T_0^{-1} (1-\varepsilon, 1) \big) + (1-p)\mu_p \big( T_1^{-1} (1-\varepsilon, 1) \big)}{(k+1)^2 \lambda \big(\big( \frac1{k+1}, \frac1{k+1-\varepsilon} \big) \big)}\\
& \ge & \lim_{\varepsilon \to 0} \frac{p \mu_p \big(\big( \frac1{k+1}, \frac1{k+1-\varepsilon} \big) \big)}{(k+1)^2 \lambda \big(\big( \frac1{k+1}, \frac1{k+1-\varepsilon} \big) \big)} \, = \, \frac{p}{(k+1)^2} h_p\Big(\frac1{k+1}\Big) >0.
\end{eqnarray*}
The fact that $h_p$ is bounded from above follows since $h_p$ is a function of bounded variation on a closed and bounded interval.
\end{proof}

The fact that the operator $\mathcal L_p$ is quasi-compact on the set of functions of bounded variation allows us to obtain a number of consequences from the Ionescu-Tulcea and Marinescu Theorem, as is done in many similar situations. The reader is referred to \cite{Bal00} for example for an outline of this approach for deterministic maps. The spectral decomposition (which is already given in \cite{Ino12}) together with Proposition~\ref{p:boundedaway} gives that 1 is a simple eigenvalue for $\mathcal L_p$, that there are no other eigenvalues on the unit circle and that $\mathcal L_p^n f \to \int f d\lambda h_p$ in $L^1(\lambda)$. This means that the system $R$ satisfies the conditions of \cite{ANV14}. (Note that $R$ is not contained in the class of random Lasota-Yorke systems discussed in Example 2.1 of \cite{ANV14}, since both $T_0$ and $T_1$ have infinitely many branches.) From \cite{ANV14} we then immediately get exponential decay of correlations:

\begin{prop}[see \cite{ANV14} Proposition 3.1]\label{p:edc}
There exist constants $C \ge 0$ and $\rho <1$, such that for all functions $f$ of bounded variation and all $g \in L^{\infty}(\lambda)$,
\[ \Big| \int_{[0,1]} \mathcal L^n_pf \cdot g d\mu_p - \int_{[0,1]} f d\mu_p \int_{[0,1]} gd\mu_p\Big| \le C \rho^n \|f \|_{BV} \| g \|_{\infty}.\]
\end{prop}

In particular, $R$ is mixing. A proof can be found in \cite{Mor85} for example.

\begin{theorem}\label{t:mixing}
The random transformation $R$ is mixing with respect to $m_p \otimes \mu_p$.
\end{theorem}

\begin{proof}
Define the function $\phi:\Omega \times [0,1] \to [0,1]$ by $\phi(\omega, x)= \omega_1 + (-1)^{\omega_1}x$. Define the cylinder sets of order 1 by
\[ [d] \times \Delta(a)_j = [d] \times \phi \Big( [d] \times\Big( \frac1{a+1}, \frac1a\Big] \Big).\]
In general the cylinders of order $k$ are given by $[d_1 \cdots d_k] \times \Delta(a_1, \ldots, a_k)_{d_1\cdots d_k}$, where
\[ \Delta(a_1, \ldots, a_k)_{d_1\cdots d_k} = \bigcap_{j=1}^k \Big(T_{d_{j-1}} \circ \cdots \circ T_{d_1}\Big)^{-1} \phi \Big( [d_j] \times\Big( \frac1{a_j+1}, \frac1{a_j}\Big] \Big). \]
Write $[\bar d] \times \Delta(\bar a)_{\bar d}^k$ for $[d_1 \cdots d_k] \times \Delta(a_1, \ldots, a_k)_{d_1\cdots d_k}$. These sets form a generating semi-algebra for the $\sigma$-algebra on $\Omega \times [0,1]$. A straightforward computation using the previous proposition now gives that for any two such cylinders $[\bar d]\times \Delta (\bar a)_{\bar d}^k$ and $[\bar c] \times \Delta(\bar b)_{\bar c}^{\ell}$,
\[ \lim_{n \to \infty} (m_p \otimes \mu_p) \big( R^{-n} \big([\bar c] \times \Delta(\bar b)_{\bar c}^{\ell} \big) \cap \big( [\bar d] \times \Delta (\bar a)_{\bar d}^k \big) \big) \hspace{5.5cm}\]
\[ \hspace{5.5cm} =(m_p \otimes \mu_p) \big([\bar c] \times \Delta(\bar b)_{\bar c}^{\ell}\big) (m_p \otimes \mu_p)\big([\bar d] \times \Delta(\bar a)_{\bar d}^{k}\big),
\]
where the convergence is in $L^1$.
\end{proof}

We also immediately obtain the Central Limit Theorem and the Large Deviation Principle. Let $\phi$ be an observable of bounded variation with $\int_{[0,1]} \phi d\mu_p =0$. Define $X_k : \Omega \times [0,1] \to \mathbb R$ by $X_k (\omega, x) = \phi \big( (T_{\omega_k} \circ \cdots \circ T_{\omega_1})x \big)$ and $S_n = \sum_{k=0}^{n-1} X_k$. The asymptotic variance $\sigma^2$ is defined by
\[ \sigma^2 = \lim_{n \to \infty} \frac1n \int_{\Omega \times [0,1]} S_n^2 \, d(m_p \otimes \mu_p).\]
In \cite{ANV14} it is proved that this limit exists in Proposition 3.2. Let $\mathcal M$ be the set of all Borel probability measures on $[0,1]$ that are absolutely continuous w.r.t.~$\lambda$ with a density of bounded variation.

\begin{theorem}[Central Limit Theorem, see \cite{ANV14} Theorem  3.5]
Let $\nu \in \mathcal M$. Then the process $\big( \frac{S_n}{\sqrt n}\big)_{n \ge 1}$ converges in law to $\mathcal N(0, \sigma^2)$ under the probability measure $m_p \otimes \nu$.
\end{theorem}

\begin{theorem}[Large Deviation Principle, see \cite{ANV14} Theorem 3.6]
Suppose $\sigma^2 >0$. Then there exists a non-negative rate function $c$, continuous, strictly convex, vanishing only at 0, such that for every $\nu \in \mathcal M$ and every sufficiently small $\varepsilon >0$,
\[ \lim_{n \to \infty} \frac1n \log \big( (m_p \otimes \nu)(S_n > n\varepsilon) \big) =-c(\varepsilon).\]
\end{theorem}
In particular these results hold for $\lambda$ and $\mu_p$.

\section{Further Properties of the continued fractions}
Here we make a few observations about the continued fractions produced by $K$.

\subsection{Counting Expansions}
It is clear that $K$ produces all semi-regular continued fractions of points $x \in [-1,1]$. Since typically each sequence $\omega \in \Omega$ characterises a unique continued fraction expansion, $K$ will generate uncountably many different continued fraction expansions for most $x$.


\begin{prop}\label{p:many}
Each $x \in [-1,1]\backslash \mathbb Q$ has uncountably many different continued fraction expansions given by $K$. Each $x \in ([-1,1] \cap \mathbb Q)\backslash \{0\}$ has countably infinitely many.
\end{prop}

\begin{proof}
To prove this we first show that $x \in ([-1,1] \cap \mathbb Q) \backslash \{0 \}$ if and only if for each $\omega \in \Omega$ there exists an $n \ge 0$ such that $\pi \big( K^n(\omega,x)\big) \in \big\{ \frac1k, -\frac1k \, : \, k \in\mathbb N\big\}$. One direction is clear, since if $x$ is irrational then $\pi(K(\omega,x))$ is irrational for any $\omega\in\Omega$. Now assume that $x \in [-1,1] \cap \mathbb Q \backslash \{0 \}$. First note that $\pi \big( K(\omega,x)\big) \in \{-1,0,1\}$ if and only if $x \in \big\{ \frac1k, -\frac1k \, : \, k \ge 1\big\}$. Write $x = \frac{r_1}{r_0}$ with $r_0,r_1 \in \mathbb Z$ and $|r_1| \le |r_0|$. If $|r_0|=|r_1|$, then the proposition is obtained with $n=0$. If not, then
\[ \pi \big( K(\omega,x)\big) = \frac{(-1)^{\omega_0}}{x} - d_1 = \frac{(-1)^{\omega_0}r_0 - r_1 d_1}{r_1} = \frac{r_2}{r_1} \in [-1,1].\]
Here $r_2 \in \mathbb Z$ and $|r_2| \le |r_1|$. If $|r_2|=|r_1|$, then we have the proposition with $n=0$. If not, then by continuing in the same manner we obtain a sequence of integers $r_0,r_1, \ldots, r_n$ with $|r_n| < \cdots < |r_1| < |r_0|$. This shows that the process must terminate after at most $r_0$ steps.

Now suppose $x \in [-1,1] \backslash \mathbb Q$. Then there is no $\omega \in \Omega$ and no $n \ge 0$, such that $\pi \big( K^n(\omega, x) \big) \in \big\{ \pm \frac1{k} \, : \, k \ge 1 \big\}$. Hence, each $\omega$ produces a unique continued fraction expansion for $x$. This gives the first part of the proposition.

For the second part of the statement, let $\frac{p}{q}$ be an arbitrary element of $([-1,1] \cap \mathbb Q)\backslash \{0\}$ with $p \in \mathbb Z$, $q \in \mathbb N$ and $0 < |p|< q$. Then for each $\omega \in \Omega$ there is an $N \le q$, such that $\pi \big( K^N (\omega, \frac{p}{q})\big) \in \big\{ \frac1k, -\frac1k \big\}$. If $\omega_n=1$ for all $n > N$, then there is some $k \ge 1$ such that
\[ \Big(d_n \Big(\omega, \frac{p}{q} \Big) \Big)_{n \ge 1} = (d_1, \ldots, d_N, k+1, 2, 2, 2, \ldots).\]
If there is a smallest integer $M\ge 2$ such that $\omega_{N+M}=0$, then
\[ \Big(d_n \Big(\omega, \frac{p}{q} \Big) \Big)_{n \ge 1} = (d_1, \ldots, d_N, k+1, \underbrace{2, \ldots, 2}_{M-1 \text{ times}}, 1).\]
Finally, if $\omega_{N+1}=0$, then
\[ \Big(d_n \Big(\omega, \frac{p}{q} \Big) \Big)_{n \ge 1} = (d_1, \ldots, d_N, k).\]
Hence, $K$ generates only countably many different expansions for $x$.
\end{proof}

Proposition~\ref{p:many} gives all the possible endings of digit sequences of rational points $\frac{p}{q}$ generated by $K$. From Figure~\ref{f:complete} it is clear however that $\frac{p}{q}$ also has expansions that are not generated by $K$. The corresponding digit sequences are
\[ (d_1, \ldots, d_N, k-1, 1), \quad (d_1, \ldots, d_N, k-1, 2, 2, 2, \ldots),\]
and for each $M \ge 2$,
\[ (d_1, \ldots, d_N, k-1, \underbrace{2, \ldots, 2}_{M-1 \text{ times}}, 1).\]
The fact that $K$ does not generate these expansions is a consequence of having defined $K$ by taking the intervals $\big( \frac1{k+1}, \frac1k \big]$ left-open and right-closed. So all points in $\mathbb [-1,1] \cap \mathbb Q$ have expansions of the form (\ref{q:cfeqn}) that are not generated by $K$, but we know exactly what the missing expansions are.

\subsection*{$\alpha$-continued fractions}
The $\alpha$-continued fraction transformation was first introduced by Nakada in \cite{Nak81}. Given $\alpha \in [0,1]$, define the map $T_{\alpha}: [\alpha-1, \alpha) \to [\alpha-1, \alpha)$ by $T_{\alpha}0=0$ and
\[T_{\alpha}x = \Big| \frac1x \Big| -  \Big\lfloor \Big| \frac1x \Big| - (\alpha -1) \Big\rfloor\]
for $x \neq 0$. First note that for $x < 0$, $T_1(1+x) = \big| \frac1x \big| - \lfloor \big| \frac1x \big| \rfloor$. Also note that
\[ \Big|\frac1x \Big| - \Big\lfloor \Big| \frac1x \Big| \Big\rfloor < \alpha \quad \Leftrightarrow \quad \Big\lfloor \Big| \frac1x \Big| -\alpha +1 \Big\rfloor = \Big\lfloor \Big| \frac1x \Big| \Big\rfloor\]
and
\[ \Big|\frac1x \Big| - \Big\lfloor \Big|\frac1x \Big| \Big\rfloor \ge \alpha \quad \Leftrightarrow \quad \Big\lfloor \Big|\frac1x \Big| -\alpha +1 \Big\rfloor = \Big\lfloor  \Big| \frac1x \Big| \Big\rfloor +1.\]

We can define a sequence $\omega = (\omega_n)_{n \ge 1} \in \Omega$ such that for each $n \ge 0$, $\pi \big(K^n (\omega,x)\big) = T^n_{\alpha}x$ as follows. Note that if $0< x < \alpha$, then either $T_0 x \in [\alpha-1, \alpha)$ in which case we set $\omega_1=0$ or $T_0 x -1 \in [\alpha-1, \alpha)$, in which case we set $\omega_1=1$. Similarly for $\alpha-1 \le x<0$ and the map $T_1 (x+1)$. Then
\[ T_{\omega_0} (x +\omega_0) - \omega_1 = T_{\alpha}x.\]
We can apply the same procedure for every $n \ge 2$ with the point $(T_{\omega_{n-1}} \circ \cdots \circ T_{\omega_0}) (x +\omega_0) - \omega_n$. One can easily show, using (\ref{q:pi-K}) and induction, that for each $x \in [\alpha-1, \alpha)$ and $n \ge 1$, one gets $\pi \big( K^n(\omega,x) \big) = T^n_{\alpha}x$.

\subsection{Restrictions on digits}
For regular continued fractions, expansions with restrictions on the digits have been thoroughly investigated. Continued fractions with bounded digits have applications in many fields, ranging from formal language theory to diophantine approximation to pseudo-random number generators, see \cite{Sha92} for a survey on this topic. For random continued fractions the following is easily observed.

\begin{prop}\label{p:subsetsN}
Let $\Lambda \subset \mathbb N$ be such that it does not miss two or more consecutive integers, i.e., if $k \not \in \Lambda$, then $k-1,k+1 \in \Lambda$. Then all numbers in $[-1,1]$ have a continued fraction expansion of the form (\ref{q:cfeqn}) with $d_n \in \Lambda$ for each $n$.
\end{prop}

\begin{proof}
From Figure~\ref{f:complete} it is clear that if we delete the branches $\frac1{x}-k$ and $-\frac1{x}-k$ for $k \in \Lambda$, then all the other branches still cover all of the interval $[-1,1]$. This means that for each $x \in [-1,1]$ we can choose an $\omega \in \Omega$ such that $d_n(\omega,x) \neq k$ for all $n \ge 1$.
\end{proof}

From this proposition it follows immediately that each $x \in [-1,1]$ has a continued fraction expansion with only even or only odd digits. Moreover, note that if we remove all branches corresponding to the odd (or even) digits, then there is no overlap in the system except on the points $\pm \frac1k$, implying that each irrational $x \in [-1,1]$ has a unique expansion with only odd (or even) digits.

Regular continued fractions with bounded digits are very well studied. It has been known since the works of Jarnik (\cite{Jar32}) and Good (\cite{Goo41}) that for each $N \ge 2$ the set of points $x \in [0,1]$ that has a regular continued fraction expansion with digits not exceeding $N$ has Hausdorff dimension strictly between 0 and 1. Many people have given estimates for the Hausdorff dimension of the set $E_{1,2}$ of points with a regular continued fraction expansion using only digits 1 and 2, with the most recent result by Pollicott and Jenkinson (\cite{JP01}) calculated the first 25 digits of the Hausdorff dimension of this set:
\[ \dim_H (E_{1,2}) = 0.5312805062772051416244686\ldots.\] 
In the random case we have the following.
\begin{prop}\label{p:1and2}
The set of points $x \in [-1,1]$ that have a continued fraction expansion of the form (\ref{q:cfeqn}) with $d_n \in \{1,2\}$ for all $n \ge 1$, has positive Lebesgue measure. In fact, it is a countable union of intervals, containing the interval $\big(\frac12,1\big]$.
\end{prop}

\begin{proof}
Let $A_{1,2}$ denote the set of points that have an expansion of the form (\ref{q:cfeqn}) with only digits 1 and 2. Then $A_{1,2}\cap \big(-\frac13,\frac13\big)=\emptyset$. First note that the points $\pm\frac1k$, $ k =2, 3$, have countably many expansions using only 1 and 2 as discussed after the proof of Proposition~\ref{p:many}. These expansions are not generated by $K$. For all other points in $A_{1,2}$ such expansions are generated by $K$. We consider these expansions, see Figure~\ref{f:1and2} for an illustration.

The first digit of any $x\in A_{1,2} \cap \big( \big(\frac13, \frac12\big] \cup \big[-\frac12, -\frac13\big) \big)$ must be 2, coming from $\omega_1=0$. For $n \ge 1$, let $S_n$ denote the map $S_n x = \big|\frac1x\big| -n$. Then for points in the set
\[ E = \bigcup_{n \ge 1} S_2^{-n}\Big(0,\frac13\Big) \cap \Big( \Big( \frac13, \frac12\Big] \cup \Big[ -\frac12, -\frac13\Big) \Big)\]
$K$ does not generate an expansion with only digits 1 and 2. Note that $S_2$ has fixed point $\sqrt 2-1$ and $|S_2' x|>1$ for all $x \in \big(\frac13, \frac12\big] \cup \big[-\frac12, -\frac13\big)$. Since
\[ \sqrt 2-1 \not \in S^{-1}_2 \Big[0,\frac13\Big) = \Big( \frac37, \frac12 \Big] \cup \Big[-\frac12, -\frac37 \Big)\]
and
\[ S^{-2}_2 \Big[0,\frac13\Big) = \Big[ \frac25, \frac7{17} \Big) \cup \Big( -\frac7{17}, \frac25 \Big],\]
it follows that $E \subseteq \big[\frac25, \frac12\big] \cup \big[-\frac12, -\frac25\big]$ and that $E$ is a countable union of intervals. For the set $\big(\frac12, 1\big] \cap \big[-1, -\frac12 \big)$ we have the following. Since
\[ S_1^{-1}\Big[0,\frac13\Big) = \Big( \frac34,1 \Big] \cup \Big[-1,-\frac34\Big) \quad \text{ and } \quad S_2^{-1}\Big(-\frac13,0\Big] = \Big[ \frac12, \frac35 \Big)\cup \Big(-\frac35,- \frac12 \Big]\]
we see that on $\big( \frac34,1 \big]$ and $\big[ -1, -\frac34 \big)$ we will have to use $S_2$ and on $\big( \frac12, \frac35 \big)$ and $\big(-\frac35,- \frac12 \big)$  we will have to use $S_1$. Note that
\[ S_1^{-1} E \cap \Big(\frac12, 1\Big] \subseteq \Big[ \frac23, \frac57 \Big].\]
Since $\frac35 < \frac23 < \frac57 < \frac34$, on all of the interval $\big[\frac12, 1\big]$ we can avoid being mapped into $E$ by making appropriate choices between $S_1$ and $S_2$. Hence
\[ A_{1,2} = \Big(\Big[ \frac13, 1\Big] \cup \Big[-1, -\frac13\Big] \Big) \backslash E. \qedhere\]
\end{proof}

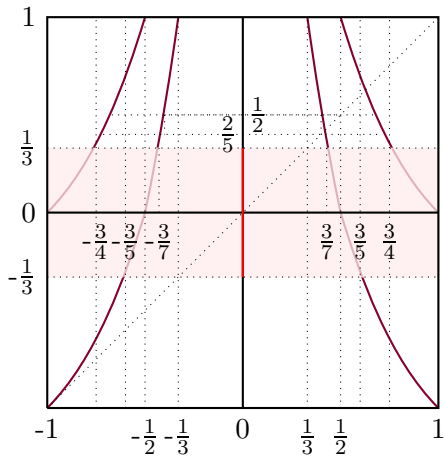
\begin{figure}[h]
\begin{center}
\begin{tikzpicture}[scale=2.6]
\draw[thick, purple!70!black] (1,0) .. controls (.75,.25) and (.6,.67) .. (.5,1);
\draw[thick, purple!70!black] (.5,0) .. controls (.42,.38) and (.36,.78) .. (.33,1);
\draw[thick, purple!70!black] (-1,0) .. controls (-.75,.25) and (-.6,.67) .. (-.5,1);
\draw[thick, purple!70!black] (-.5,0) .. controls (-.42,.38) and (-.36,.78) .. (-.33,1);
\draw[thick, purple!70!black] (1,-1) .. controls (.75,-.75) and (.6,-.33) .. (.5,0);
\draw[thick, purple!70!black] (-1,-1) .. controls (-.75,-.75) and (-.6,-.33) .. (-.5,0);
\draw[draw=red!7, fill=red!7, opacity=.75, fill opacity=.75] (-1,-.33) rectangle (1,.33);
%
\draw[dotted](-.5,-1)--(-.5,1)(-.33,-1)--(-.33,1)(.33,-1)--(.33,1)(.5,-1)--(.5,1)(-1,-1)--(1,1);
\draw[dotted](-1,-.33)--(1,-.33)(-1,.33)--(1,.33);
\draw[dotted](.75,-1)--(.75,1)(-.75,-1)--(-.75,1)(.6,-1)--(.6,1)(-.6,-1)--(-.6,1);
\draw[dotted](.4286,.33)--(.4286,0)(-.4286,.33)--(-.4286,0);
\draw[dotted](-.4,.5)--(0,.5)--(.5,.5);
\draw[dotted](-.67,.5)--(.67,.5)(-.714,.4)--(.714,.4)(.4,.5)--(.4,.4)(-.4,.5)--(-.4,.4);

\draw[thick](-1,-1)node[below]{\small -1}--(-.5,-1)node[below]{\small -$\frac12$}--(-.33,-1)node[below]{\small -$\frac13$}--(0,-1)node[below]{\small 0}--(.33,-1)node[below]{\small $\frac13$}--(.5,-1)node[below]{\small $\frac12$}--(1,-1)node[below]{\small 1}--(1,1)--(-1,1)node[left]{\small 1}--(-1,.33)node[left]{\small $\frac13$}--(-1,0)node[left]{\small 0}--(-1,-.33)node[left]{\small -$\frac13$}--(-1,-1);
\draw[thick](-1,0)--(-.75,0)node[below]{\small -$\frac34$}--(-.6,0)node[below]{\small -$\frac35$}--(-.4286,0)node[below]{\small -$\frac37$}--(.4286,0)node[below]{\small $\frac37$}--(.6,0)node[below]{\small $\frac35$}--(.75,0)node[below]{\small $\frac34$}--(1,0);
\draw[thick](0,-1)--(0,.4)node[left=-2pt]{\small $\frac25$}--(0,.5)node[right=-2pt]{\small $\frac12$}--(0,1);
\draw[thick, red](0,-.33)--(0,.33);
\end{tikzpicture}
\caption{Points that have a continued fraction expansion generated by $K$ with only digits 1 and 2 only use the maps $x \mapsto \big|\frac1{x} \big| - 1$ and $x \mapsto \big|\frac1{x} \big| - 2$.}
\label{f:1and2}
\end{center}
\end{figure}

In 1994 Lehner (\cite{Leh94}) proved that every irrational number $x$ in the interval $[1,2)$ has a unique continued fraction expansion of the form
\[ x = b_0 + \cfrac{\epsilon_0}{b_1 + \cfrac{\epsilon_1}{b_2 + \ddots}},\]
where $(\epsilon_n,b_n)$ equals either $(1,1)$ or $(-1,2)$. In \cite{DK00} Dajani and Kraaikamp established a relation between these Lehner expansions and the regular continued fraction expansions, produced by the Gauss map. Here we see that all points in $A_{1,2}$ have a Lehner-like expansion with $b_0=0$. Moreover, there are points that have more than 1 such expansions. This holds for example for all points in the interval $\big( \frac35, \frac23 \big)$ or in the interval $\big( \frac57, \frac34 \big)$.

In spite of the result from Proposition~\ref{p:1and2}, we have the following results on the asymptotic geometric and arithmetic mean of the digits, which are similar for the regular continued fractions.
\begin{prop}
For $(m_p \otimes \mu_p)$-a.e.~$(\omega,x) \in \Omega \times [0,1]$, we have
\[ 1 < \lim_{n \to \infty} \big(d_1(\omega,x)\cdots  d_n(\omega,x)\big)^{1/n} < \infty,\]
and
\[ \lim_{n \to \infty} \frac{d_1(\omega,x)+\cdots + d_n(\omega,x)}{n} = \infty.\]
\end{prop}

\begin{proof}
We use the ergodicity of the map $R$. Recall the definition of the function $b: \Omega \times [0,1] \to \mathbb N \cup \{ \infty \}$ from (\ref{q:bi}). For the first statement, we first show that $\log b \in L^1$. Suppose $h_p(x) \le M$ for all $x \in [0,1]$. Then,
\begin{eqnarray*}
\int_{\Omega \times [0,1]} \log b \, d(m_p \otimes \mu_p) &\le & M \sum_{n \ge 1} \Big[ p^2 \int_{\frac1{n+1}}^{\frac1n} \log n \, dx + p(1-p) \int_{\frac1{n+1}}^{\frac1n} \log (n+1) \, dx\\
&& + p(1-p) \int_{\frac{n-1}{n}}^{\frac{n}{n+1}} \log n \, dx + (1-p)^2 \int_{\frac{n-1}{n}}^{\frac{n}{n+1}} \log (n+1) \, dx \Big]\\
&=& M \sum_{n \ge 1} \Big[ \frac{p\log n}{n(n+1)} + \frac{(1-p)\log (n+1)}{n(n+1)} \Big] < \infty.
\end{eqnarray*}
There also is an $m >0$, such that $h_p(x) \ge m$ for all $x \in [0,1]$. Similarly as above, we obtain
\[ \int_{\Omega \times [0,1]} \log b \, d(m_p \otimes \mu_p) \ge m \sum_{n \ge 1} \Big[ \frac{p\log n}{n(n+1)} + \frac{(1-p)\log (n+1)}{n(n+1)} \Big] >0.\]
Then by the Birkhoff Ergodic Theorem, for $(m_p \otimes \mu_p)$-a.e.~$(\omega,x)$,
\[ 0 < \lim_{n \to \infty} \frac1n \sum_{i=0}^{n-1} \log b\big(R^i (\omega,x)\big) = \int_{\Omega \times [0,1]} \log b \, d(m_p \otimes \mu_p) < \infty.\]
The result now follows from (\ref{q:dandb}).

For the other statement, one first notices that $b(\omega,x) \ge \frac1x -1$ if $\omega_1=0$ and $b(\omega,x) \ge \frac1{1-x}-1$ if $\omega_1=1$. Hence, $b \not \in L^1$. Therefore, consider the functions 
\[ b_N (\omega,x) = \left\{
\begin{array}{ll}
b(\omega,x) 1_{( \frac1{N+1},1]}(x), & \text{if } \omega_1=0,\\
b(\omega,x) 1_{[0, \frac{N}{N+1})}(x), & \text{if } \omega_1=1.\\
\end{array}\right.\]
Then each $b_N$ is bounded and the sequence $\{ b_N \}$ is increasing, so by Beppo-Levi's Theorem,
\begin{equation}\label{q:bN}
\lim_{N \to \infty} \int_{\Omega \times [0,1]} b_N \, d(m_p \otimes \mu_p) = \int_{\Omega \times [0,1]} b \, d(m_p \otimes \mu_p) = \infty.
\end{equation}
Moreover, the Birkhoff Ergodic Theorem gives the existence of a $(m_p \otimes \mu_p)$-measure 1 set of $(\omega,x)$, such that
\[ \lim_{n \to \infty} \frac1n \sum_{i=0}^{n-1} b_N \big( R^i(\omega,x)\big) = \int_{\Omega \times [0,1]} b_N \, d(m_p \otimes \mu_p)\]
for all $N \ge 1$. Then by (\ref{q:bN}) we get that for $(m_p \otimes \mu_p)$-a.e.~$(\omega,x)$,
\begin{eqnarray*}
\liminf_{n \to \infty} \frac1n \sum_{i=1}^n d_i(\omega,x) &=& \liminf_{n \to \infty} \frac1n \sum_{i=0}^{n-1} b\big(R^i (0\omega,x) \big)\\
& \ge & \lim_{N \to \infty} \lim_{n \to \infty} \frac1n \sum_{i=0}^{n-1} b_N \big(R^i (0\omega,x) \big) = \int_{\Omega \times [0,1]} b_N \, d(m_p \otimes \mu_p) = \infty.
\end{eqnarray*}
This gives the result.
\end{proof}

\section{Further Questions and Comments}
So far we have established that the density $h_p$ is of bounded variation and is bounded away from zero. It satisfies the equation
 \[ \aligned
  h_p(x) = \mathcal L_p h_p(x)&=\sum_{k=1}^\infty \left[
  \frac {p}{(k+x)^2} h_p\left(\frac {1}{k+x}\right)+
  \frac {1-p}{(k+x)^2} h_p\left(1-\frac {1}{x+k}\right) \right]\\
  &= p\mathcal L_0 h_p(x)+(1-p) \mathcal L_1h_p(x),  \endaligned
\]
where $\mathcal L_0$, $\mathcal L_1$ are transfer operators for $T_0$ and $T_1$ respectively. Both operators $\mathcal L_0$ and $\mathcal L_1$ preserve cones of positive smooth (analytic) functions on $[0,1]$. Considerable effort has been put in the analysis of the spectral properties of $\mathcal L_0$, which are now well
understood, see \cite{Ios1} for a recent overview. It would be interesting to study the spectral properties of $\mathcal L_p$ in more detail. Based on simulations, we suspect the following.
\begin{conjecture}
For each $0 < p < 1$ the function $h_p$ is a \emph{smooth function} on $[0,1]$.
\end{conjecture}
The exact smoothness condition is to be determined. We presume that applying relatively standard  techniques one could strengthen the results of the present paper by showing that $h_p\in C^k([0,1])$ for all $k\in \mathbb N$. In fact, we believe that the density is $C^\infty$; however, most probably it is not real-analytic (which is the case for the Gauss map).

This conjecture is motivated as follows. In \cite{MR1,Jen1} the authors investigate the relation 
between $\mathcal L_0$ and the the integral operator $\mathcal K_0$ acting on the  Hilbert space $L^2(\mathbb R_+,\mu)$, given by
\[ \mathcal K_0 \phi(s) = \int_{0}^\infty \frac{J_1(2\sqrt{st})}{\sqrt{st}} \phi(t)\, d\mu(t), \]
where $J_1$ is the Bessel function of the first kind, and $\mu$ is the measure on $\mathbb R_+$ with the density
\[ d\mu = \frac {t}{e^{t}-1} dt.\]
The operator $\mathcal K_0$ has a symmetric kernel $K_0(s,t)=\frac {J_1(2\sqrt{st})}{\sqrt{st}}$, and has several nice properties, e.g., is nuclear. In a similar fashion, existence of a positive smooth fixed point of $\mathcal L_p$ will follow from the existence of a positive fixed point of $\mathcal K_p$
\begin{equation}\label{kernels}
\aligned
\mathcal K_p \phi(s)&=p\mathcal K_0\phi(s)+(1-p)\mathcal K_1 \phi(s) \\
&=
p\int_{0}^\infty \frac 
{J_1(2\sqrt{st})}{\sqrt{st}} \phi(t)\, d\mu(t)
+(1-p)\int_{0}^\infty \frac 
{I_1(2\sqrt{st})}{\sqrt{st}} \phi(t)e^{-t}\, d\mu(t),
\endaligned
\end{equation}
where $I_1$ is the modified Bessel function of the first kind. Technical difficulties arise from the fact that the kernel $K_1=\frac {I_1(2\sqrt{st})}{\sqrt{st}}$ albeit monotonic and positive (c.f., $K_0$ is oscillating), is not integrable.

An exact formula for the density $h_p$ would of course settle the conjecture. The most successful approach to constructing invariant densities for continued fraction transformations has been to build natural extensions of the transformations, see for example \cite{Nak81,Haa02,IS06,DKS09,KSS10,KSS12} and the references therein. This approach was also effective in determining the invariant density of the random $\beta$-transformation (see \cite{Kem14}), in which case the invariant density was not just a linear combination of the invariant densities of the two maps making up the random transformation. We have so far not been able to build a natural extension of the random continued fraction map.

With an expression for the density one could study frequencies of digits for the continued fraction expansions. It would also be interesting to consider the size of subsets of $[-1,1]$ obtained like those in Propositions~\ref{p:subsetsN} and \ref{p:1and2}. For example, for any two consecutive digits $n$ and $n+1$ is there set with positive Lebesgue measure such that all points in this set have an expansion using only these digits?

Even without a formula for the density, one could study the behaviour of $h_p$ as a function of $p$. Of particular interest is the case when $p \to 0$, since $h_0$ is unbounded and $h_p$ is of bounded variation for each $p>0$. A similar question can be asked for the metric entropy of $R$. Can one calculate this? The Shannon-McMillan-Breiman Theorem with the cylinder sets from the proof of Theorem~\ref{t:mixing} can be of help here. How does it behave as $p\to 0$? This last question could be seen as an analogue of the question of studying the entropy of $\alpha$ continued fractions as was done in \cite{CT13}.

\bibliographystyle{alpha}
\bibliography{cf1}
\end{document}